\documentclass[12 pt]{amsart}

\usepackage[margin=2.3cm]{geometry} 
\usepackage{graphicx}
\usepackage{amssymb,latexsym}
\usepackage{amsmath,amsfonts,amstext,bbm}
\usepackage[utf8]{inputenc}

\newtheorem{theorem}{Theorem}
\newtheorem{lemma}{Lemma}

\DeclareMathOperator{\M}{\mathcal{M}}
\DeclareMathOperator{\A}{\mathcal{A}}

\DeclareMathOperator{\Z}{\mathbb{Z}}

\begin{document}
	\title[The second derivative of  the maximal function]
	{The second derivative of the discrete Hardy-Littlewood maximal function}
	
	\author{Faruk Temur}
	\address{Department of Mathematics\\
		Izmir Institute of Technology, Urla, İZMİR, 35430,
		TÜRKİYE}
	\email{faruktemur@iyte.edu.tr}
	\date{April 25, 2025}

	\begin{abstract}
		The regularity of the Hardy-Littlewood maximal function, in  both discrete and continuous contexts, and for both centered and noncentered variants, has been subjected to intense study for the last two decades. But efforts so far have concentrated on   first order	differentiability and variation, as it is known that in the continuous  context higher order regularity is impossible. This short note gives the first positive result on the  higher order regularity of the  discrete noncentered maximal function. 
	\end{abstract}

	\maketitle

	\section{Introduction}\label{intro}
	
	Let $\Z^+$ denote the nonnegative integers. For a function   $f:\mathbb{Z} \rightarrow \mathbb{R}$ the discrete noncentered  averages $\A_{r,s}f(n), \ r,s\in \Z^+$ are 
	\begin{equation*}
		\A_{r,s}f(n):=\frac{1}{r+s+1}\sum_{j=-r}^{s}|f(n+j)|. 
	\end{equation*}
	Then the discrete noncentered Hardy-Littlewood  maximal function is 
	\begin{equation*}
		\mathcal{M}f(n):=\sup_{r,s\in \Z^+} \A_{r,s}f(n).
	\end{equation*} 
	It is known since the early 20th   century that this  is a bounded  operator on $l^p(\mathbb{Z}), \ 1<p\leq \infty$, and also satisfies a weak type $(1,1)$ bound. More recently  the regularity of such operators  attracted  the  interest of analysts, and a vast literature developed on this issue.  Here we will recall  only the most relevant results from  this literature,  for a broad view  the  reader may consult the survey \cite{ct}.   In discrete context regularity is studied using discrete derivatives defined by 
	\begin{equation*}
		\begin{aligned}
			f'(n)&:=f(n+1)-f(n), \\
			f''(n)&:=f(n+2)-2f(n+1)+f(n),\\
			f'''(n)&:=f(n+3)-3f(n+2)+3f(n+1)-f(n),
		\end{aligned}
	\end{equation*}
	and so on.   By the work \cite{bchp} we know that
	\[ \|(\M f)'\|_1 \leq \|f'\|_1.\] 
	Our aim in this work is to prove  the analogous result for the second derivatives when $f$ is a characteristic  function:
	
	\begin{theorem}\label{theorem char}
		Let $A\subseteq \Z$, and $1\leq p\leq \infty$. Then for the characteristic function $\chi_A$ of this set 
	\[    \|(\M \chi_A)''\|_p \leq 2^{1-\frac{1}{p}}3^{\frac{1}{p}}\|\chi_A''\|_p.    \]
	\end{theorem}

	The next section  lays the groundwork for the proof, which is presented in   the last section.
	
	\section{Preliminaries}
	
	Let  $f:\mathbb{Z} \rightarrow \mathbb{R}$. 
	By a change of variables 
	\begin{equation}\label{cov}
		\|f''\|_1=\sum_{n\in\Z}|f(n+1)+f(n-1)-2f(n)|.
	\end{equation}
	 We define the  sets of convex and concave points
	\begin{equation*}
		S_+f:=\big\{   n\in \Z: f(n+1)+f(n-1)\geq 2f(n)           \big\},  \qquad    S_-f:=\Z \setminus S_+f,
	\end{equation*}
	and the left and right boundary of concave points
	\begin{equation*}
		\begin{aligned}
			\partial_l S_-f:=\big\{   n\in S_-f: n-1   \in   S_+f      \big\}, \quad  \partial_rS_-f:=\big\{   n\in S_-f: n+1   \in   S_+f      \big\}.  
		\end{aligned}
	\end{equation*}
	The boundary  is then defined as $\partial S_-f:=	\partial_l S_-f\cup 	\partial_r S_-f$. 
	We observe that one of $S_+f,S_-f$ is empty if and only if  $\partial S_-f$ is empty. On the other hand suppose $S_+f,S_-f$ are both nonempty. Then  $\partial_l S_-f$ is empty if and only if $S_-f$  have a finite supremum, and consists of   all integers not exceeding  this supremum. Similarly  $\partial_r S_-f$ is empty if and only if  $S_-f$ have a finite infimum, and  consists   of all integers not less than this infimum.
	We define chains in $S_+f$ as sequences of consecutive elements of $S_+f$ that are of maximal length. Chains in $S_-f$ are defined analogously. For a chain  $n,n+1,\ldots, n+k$  in $S_+f$  
	\[\sum_{j=n}^{n+k}|f(j+1)+f(j-1)-2f(j)| =f(n-1)-f(n)-f(n+k)+f(n+k+1),                 \] 
	and if it is a chain in $S_-f$ 
	\[\sum_{j=n}^{n+k}|f(j+1)+f(j-1)-2f(j)| =-f(n-1)+f(n)+f(n+k)-f(n+k+1).                 \] 
	Therefore   we can bound the $l^1$ norm of the second derivative using only elements of the concave boundary and limits at infinity: 
	\begin{equation}\label{funeq}
		\begin{aligned}
			\|f''\|_1 &\leq \limsup_{n\rightarrow \infty}|f(n+1)-f(n)|  +\limsup_{n\rightarrow \infty} |f(-n-1)-f(-n)| \\  &+ 2\sum_{n\in \partial_l S_-f}f(n)-f(n-1)  + 2\sum_{n\in \partial_r S_-f}f(n)-f(n+1).
		\end{aligned}
	\end{equation}
	This observation sets the stage for our proof of Theorem 1 in the next section.

	\section{The proof of Theorem 1}

	To prove Theorem \ref{theorem char} we  use the  following  lemma,  which observes that   concavity of  $\M \chi_A$ is possible only at points of $A$.
	
	\begin{lemma}\label{Lemma char}
		Let $A$ be a subset of integers.	If  $n\in S_-\M \chi_A$, then $n\in A$. 	
	\end{lemma}
	
	\begin{proof}
		The condition $n\in S_-\M \chi_A$ implies that $\M \chi_A(n)$ is greater than one of its neighbors $\M \chi_A(n-1),\M \chi_A(n+1).$ Without loss of generality assume $\M \chi_A(n)>\M \chi_A(n-1).$ This in turn implies that $\M \chi_A(n)=\A_{0,s} \chi_A(n)$, for otherwise the strict inequality  would not hold. 
		
		Now assume to the contrary that $n\notin A$. Then 
		\[\M \chi_A(n)=\frac{1}{s+1}\sum_{j=n+1}^{n+s}\chi_A(j)=\frac{s+1}{s^2+2s+1}\sum_{j=n+1}^{n+s}\chi_A(j),\]  
		while
		\[  \M \chi_A(n-1)\geq\frac{1}{s+2}\sum_{j=n+1}^{n+s}\chi_A(j), \ \ \  \M \chi_A(n+1)\geq\frac{1}{s}\sum_{j=n+1}^{n+s}\chi_A(j),  \]
		implying
		\[  \M \chi_A(n-1)+\M \chi_A(n+1)\geq\frac{2s+2}{s^2+2s}\sum_{j=n+1}^{n+s}\chi_A(j).  \]
		Thus $2\M \chi_A(n)\leq \M \chi_A(n-1)+\M \chi_A(n+1), $ a contradiction. Hence $n\in A.$

	\end{proof}

	\begin{proof}[Proof of Theorem 1]
		If $A$ or $A^c$ is empty  the theorem is trivially true. So we may assume otherwise. We let  $p=1$, which emerges as the key case. In this case $\|\chi_A''\|_1\geq 2$.    We can clearly bound each limit term in \eqref{funeq} by 1. To bound the sum terms  we use Lemma \ref{Lemma char}: $n\in S_-\M \chi_A$ implies  $n\in A$, hence $\M \chi_A(n)=1= \chi_A(n)$. Thus  for $n$ in $\partial_l S_-\M \chi_A$
		\begin{equation*}
			\begin{aligned}
				\M \chi_A(n)-\M \chi_A(n-1)\leq  \chi_A(n)-\chi_A(n-1) \leq  2\chi_A(n)-\chi_A(n-1)-\chi_A(n+1), 
			\end{aligned}
		\end{equation*}
		and for $n$ in $\partial_r S_-\M \chi_A$
		\begin{equation*}
			\begin{aligned}
				\M \chi_A(n)-\M \chi_A(n+1)\leq  \chi_A(n)-\chi_A(n+1)  \leq  2\chi_A(n)-\chi_A(n-1)-\chi_A(n+1).
			\end{aligned}
		\end{equation*}
		Therefore we can conclude that 	
		\begin{equation*}
			\begin{aligned}
				\|(\M \chi_A)''\|_1&\leq 2+2\sum_{n\in \partial S_-\M \chi_A } |2\chi_A(n)-\chi_A(n-1)-\chi_A(n+1)|
				\leq 3 \|\chi_A''\|_1.
			\end{aligned}
		\end{equation*}
		
		For  $p> 1$, we first observe that $\|(\M \chi_A)''(n)\|_{\infty}\leq 2$.  As $\|\chi_A''\|_{\infty}\geq 1$  this immediately concludes the $p=\infty$ case. For $1<p<\infty$,
		\begin{equation*}
			\begin{aligned}
				\|(\M \chi_A)''\|_p^p	\leq 2^{p-1}\|(\M \chi_A)''\|_1\leq 3\cdot2^{p-1}  \|\chi_A''\|_1. 
			\end{aligned}
		\end{equation*}
		As  $|\chi_A''(n)|\in\{0,1,2\}$, it is dominated by  $ |\chi_A''(n)|^p.$  So 
		$ \|\chi_A''\|_1 \leq   \|\chi_A''\|_p^p$,  concluding the case.

	\end{proof}

\end{document}